\UseRawInputEncoding
\documentclass[12pt]{article}
\usepackage[centertags]{amsmath}
\usepackage{amsfonts}
\usepackage{amssymb}
\usepackage{latexsym}
\usepackage{amsthm}
\usepackage{newlfont}
\usepackage{graphicx}
\usepackage{listings}
\usepackage{booktabs}
\usepackage{abstract}
\usepackage{enumerate}
\usepackage{xcolor}
\RequirePackage{srcltx}
\date{}

\newlength{\defbaselineskip}
\newcommand{\setlinespacing}[1]%
           {\setlength{\baselineskip}{#1 \defbaselineskip}}

\newcommand{\actaqed}{\hfill $\actabox$}
{\medskip\noindent \textit{Proof of #1. }}%
{\actaqed \medskip}

\def\cA{{\mathcal A}}

\def\cC{{\mathcal C}}
\def\cD{{\mathcal D}}

\def\cK{{\mathcal K}}
\def\cL{{\mathcal L}}

\def\cS{{\mathcal S}}
\def\cT{{\mathcal T}}

\def\cV{{\mathcal V}}

\def\bbC{{\mathbb C}}

\def\bbN{{\mathbb N}}

\def\bbR{{\mathbb R}}

\def\bbT{{\mathbb T}}

\def\bF{{\mathbf F}}

\def\bH{{\mathbf H}}

\def\bW{{\mathbf W}}

\def\ba{\mathbf a}
\def\bb{\mathbf b}

\def\bj{\mathbf j}

\def\bp{\mathbf p}
\def\bq{\mathbf q}
\def\br{\mathbf r}
\def\bs{\mathbf s}
\def\btt{\mathbf t}
\def\bu{\mathbf u}

\def\bw{\mathbf w}
\def\bx{\mathbf x}
\def\by{\mathbf y}
\def\bz{\mathbf z}
 
 \def \<{\langle}
\def\>{\rangle}

\def \e{\varepsilon}

\def \ff{\varphi}
\def\al{\alpha}
\def\bt{\beta}
\def \ro{\varrho}

\def\la{\lambda}

\def \sp{\operatorname{span}}

\def\bt{\beta}

\newtheorem{Theorem}{Theorem}[section]
\newtheorem{Lemma}{Lemma}[section]
\newtheorem{Definition}{Definition}[section]
\newtheorem{Proposition}{Proposition}[section]
\newtheorem{Remark}{Remark}[section]

\newtheorem{Corollary}{Corollary}[section]
\numberwithin{equation}{section}

\newcommand{\be}{\begin{equation}}
\newcommand{\ee}{\end{equation}}

\begin{document}

\title{Nonlinear approximation with adaptive dictionaries}

\author{V. Temlyakov}

\newcommand{\Addresses}{{
  \bigskip
  \footnotesize


 \medskip
  V.N. Temlyakov, \textsc{University of South Carolina, USA,\\ Steklov Mathematical Institute of Russian Academy of Sciences, Russia;\\ Lomonosov Moscow State University, Russia; \\ Moscow Center of Fundamental and Applied Mathematics, Russia.\\  
E-mail:} \texttt{temlyakovv@gmail.com}

}}
\maketitle

\begin{abstract}{It is well known that the study of the Kolmogorov widths of a function class, which is the image of the unit ball of the $L_q$ space of an integral operator $J_K$ with the kernel $K$, is closely connected with the study of  sparse approximations of the kernel $K$ with respect to the classical bilinear dictionary. Recently, it was discovered that if instead of the Kolmogorov widths we study the errors of optimal linear sampling recovery of the same classes, then we need to study  sparse approximations of the kernel $K$ with respect to an adaptive dictionary, which is determined by the kernel $K$. In this paper we study this important problem of nonlinear approximation with respect to an adaptive dictionary. 
Also, in this paper we continue to develop the following general approach, which is related to the above nonlinear approximation problem. We study asymptotic behavior of the errors of sampling recovery not for an individual smoothness class, how it is usually done, but for the collection of classes, which are defined by integral operators with kernels coming from a given class of functions. Earlier, such approach was realized for the Kolmogorov widths and very recently for the entropy numbers.}
\end{abstract}

\section{Introduction}
\label{In}

This paper is a followup to the recent author's paper \cite{VT217}. In this paper we continue to study approximation of the multivariate functions $K(\bx,\by)$, $\bx =(x_1,\dots,x_d)$,  $\by =(y_1,\dots,y_d)$   by linear combinations   of functions of the form $u(\bx)v(\by)$. In the case, when we can choose arbitrary functions $u$ and $v$, it is a classical problem of best bilinear approximation. In the paper \cite{VT217} it was pointed out that the problem of optimal linear recovery on function classes defined by an integral operator with the kernel $K(\bx,\by)$ is closely related to the problem of 
 approximation of $K$ by linear combinations of functions of the form $u(\bx)v(\by)$ with functions $v(\by)$ defined by the kernel $K(\bx,\by)$, namely, $v(\by) = K(\bz,\by)$ with some $\bz$. This means that we approximate $K(\bx,\by)$ with respect to a dictionary, which is determined by the function $K(\bx,\by)$ itself. We call such a process -- approximation with adaptive dictionaries (see below for more details). In the paper \cite{VT217} mostly the case $d=1$, i.e. the case of functions $K(x,y)$ of two variables, was studied. In this paper we focus on the general case $d\ge 1$. 

We now proceed to the detailed presentation. Let $(\Omega,\mu)$ be a probability space. 
  By the $L_p$, $1\le p< \infty$, norm we understand
$$
\|f\|_p:=\|f\|_{L_p(\Omega,\mu)} := \left(\int_\Omega |f|^p\,d\mu\right)^{1/p}.
$$
By the $L_\infty$-norm we understand the uniform norm of continuous functions
$$
\|f\|_\infty := \sup_{\omega\in\Omega} |f(\omega)|
$$
and with some abuse of notation we occasionally write $L_\infty(\Omega)$ for the space $\cC(\Omega)$ of continuous functions on $\Omega$. We define the vector $L_{\bp}$-norm, $\bp=(p_1,\dots,p_v)$, of functions of $v$ variables $\bx=(x_1,\dots,x_v)$ as
$$
\|f(\bx)\|_{\bp} :=\|f(\bx)\|_{(p_1,\dots,p_v)} :=\|f(\bx)\|_{p_1,\dots,p_v} := \|\cdots\|f(\cdot,x_2,\dots,x_v)\|_{p_1}\cdots\|_{p_v}.
$$

We now introduce some concepts from nonlinear sparse approximation.  

The first example of sparse approximation with respect to redundant dictionaries 
 was considered by E. Schmidt in \cite{Sch}, who studied the
approximation of functions $f(x,y)$ of two variables by bilinear forms,
$$
\sum_{i=1}^mu_i(x)v_i(y),
$$ 
in $L_2([0,1]^2)$. In this case we use the following dictionary (bilinear dictionary)
\be\label{Bi1}
\Pi :=  \{u(x)v(y)\,:\, u,v\in L_2([0,1])\},
\ee
where the   functions $u$ and $v$ are functions of a single variable. 
This problem is closely connected with properties of the integral operator
$$
(J_fg)(x) := \int_0^1 f(x,y)g(y) dy
$$
 with the kernel $f(x,y)$. E. Schmidt (\cite{Sch}) gave an expansion (known as the Schmidt expansion)
\be\label{Bi2}
f(x,y) = \sum_{j=1}^\infty s_j(J_f) \phi_j(x)\psi_j(y),
\ee
where $\{s_j(J_f)\}$ is a nonincreasing sequence of singular numbers of $J_f$, i.e. $s_j(J_f) := \lambda_j(J^*_fJ_f)^{1/2}$, where $\{\lambda_j(A)\}$ is a sequence of eigenvalues of an operator $A$, and $J^*_f$ is the adjoint operator to $J_f$. The two sequences $\{\phi_j(x)\}$ and $\{\psi_j(y)\}$ form orthonormal sequences of eigenfunctions of the operators $J_fJ_f^*$ and $J_f^*J_f$, respectively. He also proved that 
$$
\left\|f(x,y) -\sum_{j=1}^m s_j(J_f) \phi_j(x)\psi_j(y)\right\|_{L_2} 
$$
\be\label{Bi3}
=
\inf_{u_j,v_j\in L_2, \quad j=1,\dots,m}\left\|f(x,y) -\sum_{j=1}^m  u_j(x)v_j(y)\right\|_{L_2}.
\ee
The reader can find a detailed discussion of this connection in \cite{VTbook}, Ch.2.

In a general setting we are working in a Banach space $X$ with a redundant system of elements $\cD$ (dictionary $\cD$).   
An element (function, signal) $h\in X$ is said to be $m$-sparse with respect to $\cD$ if
it has a representation $h=\sum_{i=1}^mc_ig_i$,   $g_i\in \cD$, $i=1,\dots,m$, where $\{c_i\}$ are real or complex numbers. The set of all $m$-sparse elements is denoted by $\Sigma_m(\cD)$. For a given element $f$ we introduce the error of best $m$-term approximation
$$
\sigma_m(f,\cD)_X := \inf_{h\in\Sigma_m(\cD)} \|f-h\|_X.
$$
We now make a comment on terminology. In the greedy approximation literature we define a dictionary $\cD$ as a system $\{g\}$  of elements $g\in X$ with the following two properties
$$
\|g\|_X\le 1 \quad\text{for all} \quad g\in \cD \quad \text{and the closure of}\, \sp(\cD) =X.
$$
The normalization condition $\|g\|_X\le 1$ is imposed for convenience. Clearly, the characteristic $\sigma_m(f,\cD)_X$ does not depend on normalization. In this paper we mostly use this characteristic. 
Let us discuss the second condition. Suppose that a system $\cS\subset X$ does not satisfy this condition. Then, instead of the Banach space $X$ we consider a subspace $X_\cS$ of $X$, which is the closure (in $X$) of $\sp(\cS)$. This makes the system $\cS$ to be a dictionary in the Banach space $X_\cS$. For this reason, we sometimes with a little abuse of exactness freely use both terms {\it system} and {\it dictionary} for a general system. In the greedy approximation theory there are theorems, which guarantee convergence of certain greedy algorithms with respect to any dictionary $\cD$ for any element $f\in X$. Clearly, in the case, when we deal with a system, we can only apply those theorems to $f\in X_\cS$.   

We stress that the bilinear dictionary $\Pi$ does not depend on a function under approximation. In this sense it is not adaptive -- we use it for approximation of all functions. It turns out that in some  problems we need to study approximation of a given kernel $K(\bx,\by)$ with respect to a dictionary, which is determined by $K$. We now give a more general definition of the bilinear dictionary  (system) and define three adaptive systems. Let $\bp=(p_1,p_2)$, $1\le p_1,p_2 \le \infty$ be given. In the case $\bp=(p,p)$ for brevity we write $p$ instead of $p,p$ in the notations. Sometimes, we drop $\bp$ from the notation. 

{\bf Bilinear dictionary $\Pi(\bp)$.}   Define
$$
\Pi(\bp):=\cL\cL(\bp) :=\{g: g(\bx,\by) = u(\bx)v(\by),\, u\in L_{p_1}(\Omega^1),\, v\in L_{p_2}(\Omega^2) \}.
$$

 $\cL\cK(\bp)$-{\bf system.} Assume that $K \in L_\bp(\Omega^1\times \Omega^2)$ satisfies the following property. 
For any $\bz \in \Omega^1$ we have $K(\bz,\cdot) \in L_{p_2}(\Omega^2)$. Define
$$
\cL\cK(\bp) :=\{g: g(\bx,\by) = u(\bx,\bz)K(\bz,\by),\,  \forall \bz \in \Omega^1\,\,\text{we have}\,\, u(\cdot,\bz) \in L_{p_1}(\Omega^1)  \}.
$$

 $\cK\cL(\bp)$-{\bf system.} Assume that $K \in L_\bp(\Omega^1\times \Omega^2)$ satisfies the following property. 
For any $\bz \in \Omega^2$ we have $K(\cdot,\bz) \in L_{p_1}(\Omega^1)$. Define
$$
\cK\cL(\bp) :=\{g: g(\bx,\by) = K(\bx,\bz)v(\bz,\by),\,  \forall \bz \in \Omega^2\,\,\text{we have}\,\, v(\bz,\cdot) \in L_{p_1}(\Omega^2)  \}.
$$

$\cK\cK(\bp)$-{\bf system.} Assume that $K \in L_\bp(\Omega^1\times \Omega^2)$ satisfies the following property. 
For any $\ba \in \Omega^1$ we have $K(\ba,\cdot) \in L_{p_2}(\Omega^2)$ and for any $\bb \in \Omega^2$ we have $K(\cdot,\bb) \in L_{p_1}(\Omega^1)$. Define
$$
\cK\cK(\bp) :=\{g: g(\bx,\by) = K(\bx,\bb)K(\ba,\by), \quad (\ba,\bb)\in  \Omega^1\times \Omega^2 \}.
$$
Note that in the literature (see \cite{BS}) the functions $K_{ab}(x,y) := K(x,b)K(a,y)$, $(x,y),(a,b) \in [0,1]^2$ are called {\it cross-functions} of the function $K(x,y)$ and the system $\cK\cK(\infty)$ is called the {\it system of cross-functions}. 

In Section \ref{RI} we prove the following general inequalities (see Theorem \ref{2InT1}): For any $b\in(1,2]$ there exists a positive constant $B=B(b)$ such that for any continuous $K$ we have
\be\label{In1}
  \sigma_m(K,\cL\cK(\infty))_\infty \le Bm^{1/2}  \sigma_{\theta (m-1)}(K,\Pi(\infty))_\infty
  \ee
  with $\theta =1/b$ in the real case and $\theta =1/(2b)$ in the complex case. 
Also, we prove there that the extra factor $m^{1/2}$ in the inequality (\ref{In1})   is sharp (see Proposition \ref{2InP1}).

A number of results (upper bounds, lower bounds, and sometimes the right orders) are obtained in the paper \cite{VT217} in the case $d=1$ for the following setting: Estimate 
$$
\sup_{K\in \bF}  \sigma_m(K,\cL\cK(\bp))_\bp  
 $$
for a certain function class $\bF$. 

In this paper we extend some of those results from the case $d=1$ to the general case $d\ge 1$. For that we apply here the same general strategy, which was used in \cite{VT217}. It is a three step strategy. First, we relate  $\sigma_m(K,\cL\cK(\bp))_\bp$ to the optimal linear recovery characteristic $\ro_m(\bW^K_q,L_\bp)$. Second, we relate  $\sigma_m(K,\Pi(\bp))_\bp$, $\bp= (p,\infty)$, to the Kolmogorov width $d_m(\bW^K_1,L_p)$. Third, we use known results, which provide an upper bound on $\ro_m(\bW,L_p)$ in terms of the Kolmogorov width $d_n(\bW,L_\infty)$. Note that the first inequality of that type, namely, the inequality
$$
\ro_{bn}(\bW,L_2) \le Bd_n(\bW,L_\infty)
$$
was obtained in \cite{VT183} (see Theorem \ref{VT183T1} below). Later, some generalizations of that inequality to the case $p\in (2,\infty]$ were proved in \cite{KKLT} and \cite{KPUU2}. Here we use   Theorem \ref{krT1}. 

In Section \ref{RI} we prove the following bound (see Corollary \ref{UBC1}). Assume that  we have $r_1>1/2$, $r_2>0$. Then (see the definition of classes $\bW^\br_2$ in Section \ref{fc} below)
 $$
\sup_{K\in \bW^\br_2}  \sigma_m(K,\cL\cK)_{2} \ll m^{-r_1-r_2}(\log m)^{(d-1)(r_1+r_2)+1/2} .
 $$

As we already pointed out above the study of linear recovery is closely related to the Kolmogorov widths. Namely, to the Kolmogorov widths in the uniform norm $L_\infty$.  In Section \ref{NK} we focus on the case of the uniform norm and complement the  results  known in the case of $L_p$, $p\in [2,\infty)$, by the case $p=\infty$. The following bound (see Theorem \ref{KBT1} below) is a step in that direction.  
Assume that  we have $r_1>1/2$, $r_2>0$. Then (see the definition of classes $\bW^\br_2$ in Section \ref{fc} below)
 $$
\sup_{K\in \bW^\br_2}  d_m(\bW^K_2)_\infty \ll m^{-r_1-r_2-1/2}(\log m)^{(d-1)(r_1+r_2)+1/2} .
 $$
 
Thus, the new results of the paper are contained in Sections \ref{NK} and \ref{RI}. In Section \ref{fc} we present the definitions of function classes that we discuss in the paper. In  Section \ref{kr} we formulate results on inequalities between the error of optimal linear sampling recovery and the Kolmogorov widths. In Section \ref{KB} we collect some of the known results on the Kolmogorov widths and their relation to the sparse approximation with respect to the bilinear dictionary $\Pi$. 

\section{Function classes}
\label{fc}

We begin with the definition of classes $\bW^\ba_\bq$ (see, for instance, \cite{VTmon}, p.31, in the case of scalar $q$).
\begin{Definition}\label{fcD1}
In the univariate case, for $a>0$, let
\be\label{Bi8}
F_a(x):= 1+2\sum_{k=1}^\infty k^{-a}\cos (kx-a\pi/2)
\ee
be the Bernoulli kernel and in the multivariate case, for $\ba=(a_1,\dots,a_v) \in \bbR^v_+$, $\bx=(x_1,\dots,x_v)\in \bbT^v$, let
\be\label{Bi8m}
F_\ba(\bx) := \prod_{j=1}^v F_{a_j}(x_j).
\ee
Denote for $\mathbf{1}\le \bq\le \infty$ (we understand the vector inequality coordinate wise)
$$
\bW^\ba_\bq := \{f:f=\varphi\ast F_\ba,\quad \|\varphi\|_\bq \le 1\},
$$
where
$$
( F_\ba \ast \varphi)(\bx):= (2\pi)^{-v}\int_{\bbT^v} F_\ba(\bx-\by) \ff(\by)d\by,\quad \bbT^v := [0,2\pi)^v.
$$
\end{Definition}
The classes $\bW^\ba_\bq$ are classical classes of functions with {\it dominating mixed derivative}
(Sobolev-type classes of functions with mixed smoothness).
 
We now proceed to the definition of the classes $\bH^\ba_\bq := \bH^{\ba,v}_\bq$ of periodic functions of $v$ variables, which is based on the mixed differences (see, for instance, \cite{VTmon}, p.31,  in the case of scalar $q$).  
 
\begin{Definition}\label{fcD2}
Let  $\btt =(t_1,\dots,t_v)$ and $\Delta_{\btt}^l f(\bx)$
be the mixed $l$-th difference with
step $t_j$ in the variable $x_j$, that is
$$
\Delta_{\btt}^l f(\bx) :=\Delta_{t_v,v}^l\cdots\Delta_{t_1,1}^l
f(x_1,\dots ,x_v) .
$$
Let $e$ be a subset of natural numbers in $[1,v]$. We denote
$$
\Delta_{\btt}^l (e) :=\prod_{j\in e}\Delta_{t_j,j}^l,\qquad
\Delta_{\btt}^l (\varnothing) := Id \,-\, \text{identity operator}.
$$
We define the class $\bH_{\bq,l}^\ba B$, $l > \|\ba\|_\infty$,   as the set of
$f\in L_\bq(\bbT^v)$ such that for any $e$
\be\label{Bi9}
\bigl\|\Delta_{\btt}^l(e)f(\bx)\bigr\|_\bq\le B
\prod_{j\in e} |t_j |^{a_j} .
\ee
In the case $B=1$ we omit it. It is known (see Theorem \ref{H} below)  that the classes $\bH^\ba_{\bq,l}$ with different $l>\|\ba\|_\infty$ are equivalent. So, for convenience we omit $l$ from the notation. 
\end{Definition}

We now formulate a result, which gives an equivalent description of classes $\bH^\ba_{\bq,l}$. 
 We need some classical trigonometric polynomials. The univariate Fej\'er kernel of order $j - 1$:
$$
\mathcal K_{j} (x) := \sum_{|k|\le j} \bigl(1 - |k|/j\bigr) e^{ikx} 
=\frac{(\sin (jx/2))^2}{j (\sin (x/2))^2}.
$$
The Fej\'er kernel is an even nonnegative trigonometric
polynomial of order $j-1$.  It satisfies the obvious relations
\be\label{FKm}
\| \mathcal K_{j} \|_1 = 1, \qquad \| \mathcal K_{j} \|_{\infty} = j.
\ee
Let $\cK_\bj(\bx):= \prod_{i=1}^v \cK_{j_i}(x_i)$ be the $v$-variate Fej\'er kernels for $\bj = (j_1,\dots,j_d)$ and $\bx=(x_1,\dots,x_v)$.

The univariate de la Vall\'ee Poussin kernels are defined as follows
$$
\cV_m := 2\cK_{2m} - \cK_m.
$$
We also need the following special trigonometric polynomials.
Let $s$ be a nonnegative integer. We define
$$
\mathcal A_0 (x) := 1, \quad \mathcal A_1 (x) := \mathcal V_1 (x) - 1, \quad
\mathcal A_s (x) := \mathcal V_{2^{s-1}} (x) -\mathcal V_{2^{s-2}} (x),
\quad s\ge 2,
$$
where $\mathcal V_m$ are the de la Vall\'ee Poussin kernels defined above.
For $\bs=(s_1,\dots,s_v)\in \bbN^v_0$ define
$$
\cA_\bs(\bx) := \prod_{j=1}^v  \cA_{s_j}(x_j),\qquad \bx=(x_1,\dots,x_v)
$$
and
$$
A_\bs(f) := \cA_\bs \ast f.
$$

The following result is known (see, for instance, \cite{VTmon}, p.32, for the scalar $q$ and \cite{VT32} for the vector $\bq$).

\begin{Theorem}\label{H} Let $f\in \bH^\ba_{\bq,l}$, $\mathbf 1 \le \bq \le \infty$. Then, for $\bs \ge \mathbf 0$
\be\label{H1}
\|A_\bs(f)\|_\bq \le C(\ba,v,l)2^{-(\ba,\bs)}.
\ee
Conversely, from (\ref{H1}) it follows that there exists $B>0$, which does not depend on $f$, such that $f\in \bH^\ba_{\bq,l}B$.
\end{Theorem}

The reader can find results on approximation properties of these classes in the books \cite{VTmon}, \cite{VTbookMA}, and \cite{DTU}.

{\bf Notations for the function classes.} In this paper we consider the case, when $v=2d$, $d\in \bbN$, $\mathbf{1}\le \bq \le \infty$, and $\ba$ has a special form: $a_j = r_1$, $a_{j+d}= r_2$ for $j=1,\dots,d$. In this case we write $\bW^{\br}_\bq = \bW^{(\br^1,\br^2)}_\bq= \bW^{r_1,r_2}_\bq$ and $\bH^{\br,2d}_\bq = \bH^{(\br^1,\br^2),2d}_\bq= \bH^{r_1,r_2,2d}_\bq$, where $\br^i := (r_i,\dots,r_i) \in \bbR^d$, $i=1,2$.  Sometimes for brevity we omit $2d$ in the notation for the $\bH$ classes and write, for instance, 
$\bH^{r_1,r_2}_\bq$ instead of $\bH^{r_1,r_2,2d}_\bq$.

\section{Some known results on sampling recovery}
\label{kr}

 In this paper we study the case, when the asymptotic characteristic is the error of sampling recovery. Recall the setting 
 of the optimal linear recovery introduced in \cite{VT51}. For a fixed $m$ and a set of points  $\xi:=\{\xi^j\}_{j=1}^m\subset \Omega$, let $\Phi $ be a linear operator from $\bbC^m$ into $L_p(\Omega,\mu)$.
Denote for a class $\bF$ (usually, centrally symmetric and compact subset of $L_p(\Omega,\mu)$)
$$
\varrho_m(\bF,L_p) := \inf_{\xi} \inf_{\text{linear}\, \Phi } \sup_{f\in \bF} \|f-\Phi(f(\xi^1),\dots,f(\xi^m))\|_p.
$$
The above described recovery procedure is a linear procedure.

Most of the known results on optimal sampling recovery deal with the linear recovery methods. We now give some very brief comments on recent results in this direction and 
refer the reader to the books \cite{DTU}, \cite{VTbookMA} and to the survey paper \cite{KKLT} for a discussion of the previous results in this direction. We are interested in results, which relate the errors of sampling recovery with the Kolmogorov widths for general function classes. We begin with a result from \cite{VT183}. 

\begin{Theorem}[{\cite{VT183}}]\label{VT183T1} There exist two positive absolute constants $b$ and $B$ such that for any   compact subset $\Omega$  of $\bbR^d$, any probability measure $\mu$ on it, and any compact subset $\bF$ of $\cC(\Omega)$ we have
\be\label{kr1}
\ro_{bn}(\bF,L_2(\Omega,\mu)) \le Bd_n(\bF,L_\infty).
\ee
\end{Theorem}

The following generalization of Theorem \ref{VT183T1} to the case $2<p\le \infty$ was obtained in \cite{KPUU2}. 

\begin{Theorem}[{\cite{KPUU2}}]\label{KPUU2} Let $2\le p\le \infty$. There exists a positive absolute constant $C$ such that for any   compact subset $\Omega$  of $\bbR^d$, any probability measure $\mu$ on it, and any compact subset $\bF$ of $\cC(\Omega)$ we have
\be\label{kr2}
\ro_{4n}(\bF,L_p(\Omega,\mu)) \le Cn^{1/2-1/p}d_n(\bF,L_\infty).
\ee
\end{Theorem}

In our applications the following analog of the inequality (\ref{kr1}), which is contained in Theorem \ref{KPUU2}, 
\be\label{kr3}
\ro_{bn}(\bF,L_\infty) \le Bn^{1/2}d_n(\bF,L_\infty)
\ee
plays a fundamental role. For this reason, we now present a detailed discussion of this inequality. 

Let as above $\Omega$ be a compact subset of $\bbR^d$ and $X_N$ be an $N$-dimensional subspace of the space of continuous functions
$\cC(\Omega)$.   Given
a fixed $m$ and a set of points $\xi^1,\ldots,\xi^m\in\Omega$, we 
associate with a function $f\in \cC(\Omega)$ a vector (sample vector)
$$
S(f,\xi) := (f(\xi^1),\dots,f(\xi^m)) \in \bbC^m.
$$
We also consider the discrete norms
$$
\|S(f,\xi)\|_p:= \left(\frac{1}{m}\sum_{j=1}^m |f(\xi^j)|^p\right)^{1/p},\quad 1\le p<\infty,
$$
and $\|S(f,\xi)\|_\infty := \max_{j}|f(\xi^j)|$.

For a positive weight $\bw:=(w_1,\dots,w_m)\in \bbR^m$ consider the following seminorm
$$
\|S(f,\xi)\|_{p,\bw}:= \left(\sum_{j=1}^m w_j |f(\xi^j)|^p\right)^{1/p},\quad 1\le p<\infty.
$$
Define the best approximation of $f\in L_p(\Omega,\mu)$, $1\le p\le \infty$, by elements of $X_N$ as follows
$$
d(f,X_N)_p := \inf_{u\in X_N} \|f-u\|_p.
$$

Theorem~\ref{BT1} below was proved in \cite{VT183} under the following assumptions.

{\bf A1. Discretization.} Let $1\le p\le \infty$. Suppose that
$\xi:=\{\xi^j\}_{j=1}^m\subset \Omega$ provides the following discretization property: For any 
$u\in X_N$ in the case $p<\infty$ we have
$$
\|u\|_p \le D \|S(u,\xi)\|_{p,\bw}  
$$
and in the case $p=\infty$ we have
$$
\|u\|_\infty \le D \|S(u,\xi)\|_{\infty}  
$$
with some positive constant $D$.

{\bf A2. Weights.} Suppose that there is a positive constant
$W$ such that 
$\sum_{j=1}^m w_j \le W$.

Consider the following well known recovery operator (algorithm)  
$$
\ell p\bw(\xi)(f) := \ell p\bw(\xi,X_N)(f):=\text{arg}\min_{u\in X_N}
\|S(f-u,\xi)\|_{p,\bw},\quad 1\le p<\infty,
$$
$$
\ell \infty(\xi)(f) := \ell \infty(\xi,X_N)(f):=\text{arg}\min_{u\in X_N}
\|S(f-u,\xi)\|_{\infty}.
$$
Note that the above algorithm $\ell p\bw(\xi)$ only uses the function values $f(\xi^j)$, $j=1,\dots,m$. In the case $p=2$ it is a linear algorithm -- orthogonal projection with respect 
to the  {seminorm} $\|\cdot\|_{2,\bw}$. Therefore, in the case $p=2$ the approximation error in the $L_q$ norm by the algorithm $\ell 2\bw(\xi)$ gives an upper bound for the recovery characteristic $\ro_m(\cdot, L_q)$. 

\begin{Theorem}[{\cite[Theorem~2.1]{VT183}}]\label{BT1}  Under assumptions {\bf A1} and {\bf A2} for any $f\in \cC(\Omega)$ we have
for $1\le p<\infty$
$$
\|f-\ell p\bw(\xi)(f)\|_p \le (2DW^{1/p} +1)d(f, X_N)_\infty.
$$
Under assumption {\bf A1} for any $f\in \cC(\Omega)$ we have
$$
\|f-\ell \infty(\xi)(f)\|_\infty \le (2D+1)d(f, X_N)_\infty.
$$
\end{Theorem}

The following version of Theorem~\ref{BT1} for the error of $\|f-\ell p\bw(\xi)(f)\|_{\infty}$ under an extra condition on the Nikol'skii inequality for the $X_N$ was proved in \cite{LMT}. For completeness we present that proof here. For the reader's convenience we recall the classical definition of the Nikol'skii inequality. 

{\bf Nikol'skii-type inequalities.} Let $1\le p\le q\le\infty$ and $X_N\subset L_q(\Omega, \mu)$. The inequality
\begin{equation}\label{I4}
\|f\|_q \leq M\|f\|_p,\   \ \forall f\in X_N
\end{equation}
is called the Nikol'skii inequality for the pair $(p,q)$ with the constant $M$. 
We will also use the brief form of this fact: $X_N \in NI(p,q,M)$. Typically, $M$ depends on $N$, for instance, $M$ can be of order $N^{\frac{1}{p}-\frac{1}{q}}$.

\begin{Theorem}[{\cite{LMT}}]\label{BT1a} Let $1\le p<\infty$.  Under assumptions {\bf A1}, {\bf A2}, and an extra assumption $X_N\in NI(p,\infty,M)$   for any $f\in \cC(\Omega)$ we have
 $$
\|f-\ell p\bw(\xi)(f)\|_{\infty} \le (2MD W^{1/p} +1)d(f, X_N)_\infty.
$$
 \end{Theorem}

\begin{proof} The proof is simple and goes along the lines of the proof of Theorem~\ref{BT1}. Let $u:= \ell p\bw(\xi)(f)$. For an arbitrary $g \in X_N$ we have the following chain of inequalities.
$$
\|f-u\|_\infty \le \|f-g\|_\infty + \|g-u\|_\infty \le \|f-g\|_\infty +M\|g-u\|_p
$$
$$
\le  \|f-g\|_\infty + MD\|S(g-u,\xi)\|_{p,\bw}
$$
$$
\le  \|f-g\|_\infty + MD(\|S(f-g,\xi)\|_{p,\bw}+ \|S(f-u,\xi)\|_{p,\bw})
$$
$$
\le  \|f-g\|_\infty + 2MD\|S(f-g,\xi)\|_{p,\bw}
$$
$$
\le  \|f-g\|_\infty + 2MDW^{1/p}\|S(f-g,\xi)\|_{\infty} \le (1+ 2MDW^{1/p})\|f-g\|_\infty.
$$
Minimizing over $g\in X_N$, we complete the proof. 

\end{proof}

We now explain how to derive inequality (\ref{kr3}) from Theorem \ref{BT1a}. For a given function class $\bF \subset \cC(\Omega)$ and any $\delta>0$ find a subspace $X_N := X_N^\delta$ such that for any $f\in \bF$ we have
\be\label{kr4}
d(f, X_N)_\infty \le d_N(\bF,L_\infty) +\delta.
\ee
We want to apply Theorem \ref{BT1a} to the subspace $X_N$. We will do that for $p=2$. For that we need to check that the conditions of that theorem are satisfied. Namely, assumptions {\bf A1}, {\bf A2}, and the assumption $X_N\in NI(p,\infty,M)$. To satisfy those conditions we can choose the measure $\mu$, points $\xi^1, \dots, \xi^m$, and weights $\bw$. We begin with the measure $\mu$. We use the following fundamental result of 
  J. Kiefer and J. Wolfowitz \cite{KW}, which
 guarantees that for any finite
dimensional subspace $X_N$ of $\cC(\Omega)$ there exists a probability measure
$\mu$ on $\Omega$ such that for all $f\in X_N$ we have
\be\label{KW}
\|f\|_\infty \le N^{1/2}\|f\|_{L_2(\Omega,\mu)}.
\ee
In other words, for any subspace $X_N$ of $\cC(\Omega)$ we have $X_N \in
NI(2,\infty, N^{1/2})$ with some probability measure $\mu$. We take this measure $\mu$ and solve the discretization problem for the $L_2(\Omega,\mu)$ norm on the subspace $X_N$. 

We  use a result on discretization  from \cite{LT} (see Theorem 3.3 there), which is a generalization to the complex case of an earlier result from \cite{DPSTT2} established for the real case. 

\begin{Theorem}[{\cite{LT}}]\label{BT2} If $X_N$ is an $N$-dimensional subspace of the complex $L_2(\Omega,\mu)$, then there exist three absolute positive constants $C_1$, $c_0$, $C_0$,  a set of $m\leq   C_1N$ points $\xi^1,\ldots, \xi^m\in\Omega$, and a set of nonnegative  weights $\lambda_j$, $j=1,\ldots, m$,  such that
\be\label{kr5}
c_0\|f\|_2^2\leq  \sum_{j=1}^m \lambda_j |f(\xi^j)|^2 \leq  C_0\|f\|_2^2,\  \ \forall f\in X_N.
\ee
\end{Theorem}

For our application we need to satisfy the assumption {\bf A2} on weights. We use the following remark from \cite{VT183}.
\begin{Remark}[{\cite{VT183}}]\label{BR1} Considering a new subspace $X_N' := \{f\,:\, f= g+c, \, g\in X_N,\, c\in \bbC\}$ and applying Theorem \ref{BT2} to 
the $X_N'$ with $f=1$ ($g=0$, $c=1$) we conclude that a version of Theorem \ref{BT2} holds with the inequality $m\le C_1N$ replaced by $m\le C_1(N+1)$ and with weights satisfying 
$$
\sum_{j=1}^m \lambda_j \le C_0.
$$
\end{Remark}

We apply Theorem \ref{BT1a} with $p=2$, $M=N^{1/2}$, $D= c_0^{-1/2}$, $\bw=(\la_1,\dots,\la_m)$, $W=C_0^{1/2}$ and obtain 
\be\label{kr6}
\ro_{C_1(N+1)}(\bF,L_\infty) \le (2N^{1/2}(C_0/c_0)^{1/2}+1)(d_N(\bF,L_\infty) +\delta),
\ee
which implies (\ref{kr3}). 

In the inequality (\ref{kr3}) we only say that the parameter $b$ can be chosen as an absolute constant. There are results on the inequality (\ref{kr1}) with $b$ being arbitrarily close to $1$. The first result in that direction was proved in \cite{LT}. 

\begin{Theorem}[{\cite{LT}}]\label{SRb}  For any $b\in(1,2]$ there exists a positive constant $B=B(b)$ such that for any   compact subset $\Omega$  of $\bbR^d$, any probability measure $\mu$ on it, and any compact subset $\bF$ of $\cC(\Omega)$ we have in the real case
$$
\ro_{\lceil b(n+1) \rceil}(\bF,L_2(\Omega,\mu)) \le Bd_n(\bF,L_\infty)
$$
and in the complex case
$$
\ro_{\lceil b(2n+1) \rceil}(\bF,L_2(\Omega,\mu)) \le Bd_n(\bF,L_\infty).
$$
\end{Theorem}

In the same way as we obtained above an analog (\ref{kr3}) of the original inequality (\ref{kr1}) we can obtain the following 
analog of Theorem \ref{SRb}. 

\begin{Theorem}\label{krT1}  For any $b\in(1,2]$ there exists a positive constant $B=B(b)$ such that for any   compact subset $\Omega$  of $\bbR^d$  and any compact subset $\bF$ of $\cC(\Omega)$ we have in the real case
$$
\ro_{\lceil b(n+1) \rceil}(\bF,L_\infty) \le Bn^{1/2}d_n(\bF,L_\infty)
$$
and in the complex case
$$
\ro_{\lceil b(2n+1) \rceil}(\bF,L_\infty)) \le Bn^{1/2}d_n(\bF,L_\infty).
$$
\end{Theorem}

We complete this section with a brief historical comment. 

 {\bf Historical comments on weighted discretization.}
 In the case of weighted discretization, namely, when instead of $\frac{1}{m}\sum_{j=1}^m |f(\xi^j)|^2$ we use the weighted sum $\sum_{j=1}^m\lambda_j |f(\xi^j)|^2$, 
the problem of discretization is solved in the sense of order  in the case of real subspaces $X_N$. It is pointed out in \cite{VT159} that the paper by J. Batson, D.A. Spielman, and N. Srivastava \cite{BSS}  basically solves the discretization problem with weights.  We present an explicit formulation of this important result in our notation.

\begin{Theorem}[{\cite[Theorem 3.1]{BSS}}]\label{BSS} Let  $\Omega_M=\{x^j\}_{j=1}^M$ be a discrete set with the probability measure $\mu_M(x^j)=1/M$, $j=1,\dots,M$, and let $X_N$ be an $N$-dimensional subspace of real functions defined on $\Omega_M$.
Then for any
number $b>1$ there  exists a set of weights $\lambda_j\ge 0$ such that $|\{j: \lambda_j\neq 0\}| \le \lceil bN \rceil$ so that for any $f\in X_N $ we have
\begin{equation*}
\|f\|_2^2 \le \sum_{j=1}^M \lambda_jf(x^j)^2 \le \frac{b+1+2\sqrt{b}}{b+1-2\sqrt{b}}\|f\|_2^2.
\end{equation*}
\end{Theorem}
As observed  in \cite [Theorem 2.13]{DPTT},  this last theorem with  a general probability space $(\Omega,  \mu)$ in place of the discrete space  $(\Omega_M, \mu_M)$ remains true (with other constant in the right hand side) if $X_N\subset L_4(\Omega,\mu)$.  It was proved in \cite{DPSTT2}  that the additional assumption $X_N\subset L_4(\Omega,\mu)$ can be dropped as well.
\begin{Theorem}[{\cite[Theorem 6.3]{DPSTT2}}]\label{DPSTT} If $X_N$ is an $N$-dimensional subspace of the real $L_2(\Omega,\mu)$, then for any $b\in (1,2]$, there exist a set of $m\leq    \lceil bN \rceil$ points $\xi^1,\ldots, \xi^m\in\Omega$ and a set of nonnegative  weights $\lambda_j$, $j=1,\ldots, m$,  such that
\[ \|f\|_2^2\leq  \sum_{j=1}^m \lambda_j f(\xi^j)^2 \leq  \frac C{(b-1)^2}  \|f\|_2^2,\  \ \forall f\in X_N,\]
where $C>1$ is an absolute constant.
\end{Theorem}

\section{Some connections between the Kolmogorov widths and bilinear approximations}
\label{KB}

In this section we discuss the best $m$-term bilinear approximations in $L_\bp(\bbT^{2d})$  of functions from different classes. Our standard notation for the best $m$-term bilinear approximations is the following (see Section \ref{In})
$$
\sigma_m(\bF,\Pi)_\bp := \sup_{f\in \bF}\sigma_m(f,\Pi)_\bp.
$$
Note, that in a number of papers on this topic the following notation is used as well
$$
\tau_m(\bF)_\bp := \sigma_m(\bF,\Pi)_\bp,\qquad \tau_m(K)_\bp := \sigma_m(K,\Pi)_\bp.
$$
In the formulation of the known results we use the $\tau$ notation, which is used in the corresponding papers. 

We begin with a simple lemma, which was proved (in a particular case) in \cite{VT32}. For completeness we present 
a proof here.

\begin{Lemma}[{\cite{VT32}}]\label{KBL1} We have for $1\le p \le \infty$
\be\label{KB1}
d_n(\bW^K_1,L_p) = \tau_n(K)_{p,\infty}
\ee
and for $1\le q \le \infty$
\be\label{KB1a}
d_n(\bW^K_q,L_p) \le \tau_n(K)_{p,q'}.
\ee
\end{Lemma}
\begin{proof} It is clear that it is sufficient to prove (\ref{KB1}) for continuous functions $K$. For a fixed $\by \in \Omega^2$
the function $K(\bx,\by)$ as a function on $\bx\in \Omega^1$ belongs to the closure of the class $\bW^K_1$. Therefore,
\be\label{KB2}
d_n(\bW^K_1,L_p) \ge \inf_{u_i,v_i}\left \|K(\bx,\by) -\sum_{i=1}^n u_i(\bx)v_i(\by)\right\|_{p,\infty} = \tau_n(K)_{p,\infty}.
\ee
We now prove (\ref{KB1a}). Let for $\e>0$ the systems of functions $\{u_i\}_{i=1}^n \subset L_p(\Omega^1)$ and $\{v_i\}_{i=1}^n \subset L_\infty(\Omega^2)$ be such that
\be\label{KB3}
\left \|K(\bx,\by) -\sum_{i=1}^n u_i(\bx)v_i(\by)\right\|_{p,q'} \le \tau_n(K)_{p,q'} +\e.
\ee
Then for any $\ff \in L_q(\Omega^2)$, $\|\ff\|_q \le1$, we have  
\be\label{KB4}
\int_{\Omega^2} \left(K(\bx,\by) -\sum_{i=1}^n u_i(\bx)v_i(\by)\right)\ff(\by)d\mu_2 = f(\bx) - \sum_{i=1}^n a_iu_i(\bx)
\ee
and 
$$
\left\|f(\bx) - \sum_{i=1}^n a_iu_i(\bx)\right\|_p \le \int_{\Omega^2} \left\|K(\cdot,\by) -\sum_{i=1}^n u_i(\cdot)v_i(\by)\right\|_p|\ff(\by|d\mu_2
$$
$$
\le \left\|K(\bx,\by) -\sum_{i=1}^n u_i(\bx)v_i(\by)\right\|_{p,q'} \le \tau_n(K)_{p,q'} +\e.
$$
This implies that
\be\label{KB5}
d_n(\bW^K_q,L_p) \le \tau_n(K)_{p,q'},
\ee
which proves (\ref{KB1a}). 
Inequalities (\ref{KB2}) and (\ref{KB5}) with $q=1$ complete the proof of (\ref{KB1}). 
\end{proof}

{\bf Some useful tricks.} 
For bounded linear operators $P\,:\, X\to Y$ and $Q\,:\, Y\to Z$ acting in Banach spaces $X$, $Y$, $Z$ we have the following simple inequality
\be\label{KB6}
d_{2n}(QP(B_X),Z) \le d_n(P(B_X),Y)d_n(Q(B_Y),Z). 
\ee

Let $H$ be a Hilbert space and $J\,:\, H\to H$ be a compact linear operator. Then 
 \be\label{KB7}
 d_n(J(B_H),H) = s_{n+1}(J). 
 \ee

Let $K\in L_2(\Omega^1\times\Omega^2)$. Then the following Schmidt's formula holds
\be\label{KB8}
\tau_n(K)_2 = \left(\sum_{i=n+1}^\infty s_i(J_K)^2\right)^{1/2},
\ee
which implies that 
\be\label{KB9}
s_{2n}(J_K) \le n^{-1/2}\tau_n(K)_2.
\ee

{\bf Some  known results on bilinear approximation and singular numbers.}

{\bf The case $d\ge 1$.} Here is the result from \cite{VT47}.

\begin{Theorem}[{\cite{VT47}}, Theorem 2.1]\label{VT47T2.1} Let $\br = (r_1,\dots,r_1,r_2,\dots,r_2)\in \bbR_+^{2d}$ have the first 
$d$ coordinates equal $r_1$ and the rest equal $r_2$. Assume that $r_i>1/2$, $i=1,2$ and $2\le \bq \le \infty$, $2\le \bp<\infty$. Then for the class $\bW^\br_{\bq}$ we have
$$
\tau_m(\bW^\br_{\bq})_\bp \asymp m^{-r_1-r_2} (\log m)^{(r_1+r_2)(d-1)}.
$$
\end{Theorem}

\begin{Corollary}\label{KBC1} Under conditions of Theorem \ref{VT47T2.1} we have
$$
s_m(\bW^\br_{\bq}) := \sup_{K\in \bW^\br_{\bq}} s_m(J_K) \ll m^{-r_1-r_2-1/2} (\log m)^{(r_1+r_2)(d-1)}.
$$
\end{Corollary}

Here are the corresponding results for the $\bH$ classes form \cite{VT47}.

\begin{Theorem}[{\cite{VT47}}, Theorem 2.2]\label{VT47T2.2} Let $\br = (r_1,\dots,r_1,r_2,\dots,r_2)\in \bbR_+^{2d}$ have the first 
$d$ coordinates equal $r_1$ and the rest equal $r_2$. Assume that $r_i>1/2$, $i=1,2$ and $2\le \bq \le \infty$, $2\le \bp<\infty$. Then for the class $\bH^\br_{\bq}$ we have
$$
\tau_m(\bH^\br_{\bq})_\bp \asymp m^{-r_1-r_2} (\log m)^{(r_1+r_2+1)(d-1)}.
$$
\end{Theorem}

\begin{Corollary}[{\cite{VT47}}, Theorem 3.1]\label{KBC2} Let $\br = (r_1,\dots,r_1,r_2,\dots,r_2)\in \bbR_+^{2d}$ have the first 
$d$ coordinates equal $r_1$ and the rest equal $r_2$. Assume that $r_i>0$, $i=1,2$ and $2\le \bq \le \infty$. Then for the class $\bH^\br_{\bq}$ we have
$$
s_m(\bH^\br_{\bq}) \asymp m^{-r_1-r_2-1/2} (\log m)^{(r_1+r_2+1)(d-1)}.
$$
\end{Corollary}

\begin{Remark}\label{KBR1} In the above Theorems \ref{VT47T2.1} and \ref{VT47T2.2} we impose the restriction $r_i>1/2$, $i=1,2$. 
We need this restriction for proving the upper bounds in the case of scalar $q=2$ and arbitrarily large $p$. In the case $p=q=2$ it is sufficient to assume that $r_i>0$, $i=1,2$. In the case $r_1=r_2$ Theorems \ref{VT47T2.1} and \ref{VT47T2.2} were proved in \cite{VT35}. 
\end{Remark}

{\bf The case $d=1$.}
The case $d=1$ is better studied than the general case. We now formulate the corresponding results.
The following results are from \cite{VT32}. We use the following notation for $1\le q,p\le\infty$   
 \be\label{ksi}
 \xi(q,p):= \left(\frac{1}{q} - \max\left(\frac{1}{2},\frac{1}{p}\right)\right)_+, \quad  (a)_+ :=\max(a,0).
\ee
 
 \begin{Theorem}[{\cite{VT32}}, Theorem 2]\label{BiT4} Let $d=1$ and $\bF^\br_\bq$ denote one of the classes $\bW^\br_\bq$ or $\bH^\br_\bq$. Then for $\br > \mathbf{1}$ and $1\le q_1\le p_1 \le \infty$, $1\le q_2,p_2 \le \infty$ we have 
 $$
  \tau_m(\bF^\br_\bq)_\bp \asymp m^{-r_1-r_2 + \xi(q_1,p_1)}.  
 $$
 \end{Theorem}

\begin{Remark}\label{BiLB} Note that Theorem \ref{BiT4} is proved in \cite{VT32} under weaker conditions on $\br$ than above. That restriction on $\br$ is needed for the proof of the upper bounds. For the lower bounds it is sufficient to assume that $\br > (1/q_1-1/p_1, (1/q_2-1/p_2)_+)$.
\end{Remark}

 \begin{Corollary}[{\cite{VT32}}, Theorem 3.2]\label{BiC3} Under conditions of Theorem \ref{BiT4} we have 
 $$
 \sup_{K\in \bF^\br_\bq} s_m(J_K) \asymp m^{-r_1-r_2 +  \max\left(\frac{1}{2},\frac{1}{q_1}\right)-1}.
   $$
 \end{Corollary}

{\bf Some known results on the Kolmogorov widths of classes $\bW^K_q$.}

We begin with the case of univariate functions ($d=1$), in which case the kernel $K$ is a function on two variables. 
 The following  results are proved in \cite{VT32}.
 
  \begin{Theorem}[{\cite{VT32}}, Theorem 4.2]\label{BiT5} Let $d=1$ and $\bF^\br_1$ denote one of the classes $\bW^\br_1$ or $\bH^\br_1$. Then for $1\le q,p \le \infty$ and $\br > (1,1+\max(1/2,1/q))$   we have 
 $$
\sup_{K\in \bF^\br_1}  d_m(\bW^K_q)_p \asymp m^{-r_1-r_2 + \xi(q,p)} 
 $$
 with $\xi(q,p)$ defined in (\ref{ksi}).
 \end{Theorem}
 
 For $\bq =(q_1,q_2)$, $\bp=(p_1,p_2)$, $1\le q_1 \le p_1 \le \infty$, $1\le q_2,p_2 \le \infty$ denote
 $$
 \br(\bq,\bp) := \begin{cases} (1/q_1-1/p_1,(1/q_2-1/p_2)_+), & 1\le q_1\le p_1 \le 2, \\
 (1/q_1,1/q_2), & 2\le q_1\le p_1\le \infty, p_1>2,\\
 (1/q_1,\max(1/2,1/q_2)), & 1\le q_1<2< p_1 \le \infty.
 \end{cases}
 $$
 
   \begin{Theorem}[{\cite{VT32}}, Theorem 4.1]\label{BiT5a} Let $d=1$ and $\bF^\br_\bq$ denote one of the classes $\bW^\br_\bq$ or $\bH^\br_\bq$. Then for $\bp=(p,\infty)$, $1\le q_1 \le p \le \infty$, $1\le q_2\le \infty$  and $\br > \br(\bq,\bp)$   we have 
 $$
\sup_{K\in \bF^\br_\bq}  d_m(\bW^K_1)_p \asymp m^{-r_1-r_2 + \xi(q_1,p)} 
 $$
 with $\xi(q,p)$ defined in (\ref{ksi}).
 \end{Theorem}
 
 Note that in the case $\br > \mathbf{1}$ Theorem \ref{BiT5a} follows from Lemma \ref{KBL1} and Theorem \ref{BiT4}
 (see also Remark \ref{BiLB}).

 In the above Theorem \ref{BiT5} we consider the case of classes $\bF^\br_1$. Some results on the classes $\bF^\br_\bq$
 are obtained in \cite{VT47} (see Theorem 3.1' there). We formulate that result as Theorem \ref{BiT5b} and
 refer the reader to the paper \cite{VT32} for further results and historical comments on bilinear approximation of functions on two variables with mixed smoothness. Denote $$\br(\bq) := ((1/q_1-1/2)_+,(1/q_2-1/2)_+).$$

  \begin{Theorem}[{\cite{VT47}}, Theorem 3.1']\label{BiT5b} Let $d=1$ and $\bF^\br_\bq$ denote one of the classes $\bW^\br_\bq$ or $\bH^\br_\bq$, $\mathbf{1} \le \bq \le \infty$. Then for $2\le a\le \infty$, $1\le b\le \infty$ under assumption that $\br > \br(\bq)$ for $1\le b\le 2$ and $\br > \br(\bq)+ (1/2,0)$ for $b>2$ we have 
 $$
\sup_{K\in \bF^\br_\bq}  d_m(\bW^K_a)_b \asymp m^{-r_1-r_2 + \max(1/q_1,1/2) -1} .
 $$
 \end{Theorem}

Here is an analog of Theorem \ref{BiT5}, which holds for $d=1$, in the case $d>1$. 

\begin{Theorem}[{\cite{VT47}}, Theorem 3.2]\label{BiT5d} Let $d\in\bbN$ and $\mathbf{2} \le \bq \le \infty$, $2\le a <\infty$, $1<b<\infty$. 
Assume that in the case $b\in (1,2]$ we have $r_i>0$, $i=1,2$, and in the case $b\in (2,\infty)$ we have $r_1>1/2$, $r_2>0$. Then
 $$
\sup_{K\in \bW^\br_\bq}  d_m(\bW^K_a)_b \asymp (m^{-1}(\log m)^{d-1})^{r_1+r_2}m^{-1/2} .
 $$
 \end{Theorem}

\section{Some new results on the Kolmogorov widths}
\label{NK}

Theorems \ref{VT183T1} and \ref{KPUU2} show that in the study of linear recovery the Kolmogorov widths in the uniform norm $L_\infty$ play an important role. In this section we focus on the case of the uniform norm and complement the  results  known in the case of $L_p$, $p\in [2,\infty)$, by the case $p=\infty$. The following Theorem \ref{KBT1} is a step in that direction from the above Theorem \ref{BiT5d}.

 \begin{Theorem}\label{KBT1} Let $d\in\bbN$. 
Assume that  we have $r_1>1/2$, $r_2>0$. Then
 $$
\sup_{K\in \bW^\br_2}  d_m(\bW^K_2)_\infty \ll m^{-r_1-r_2-1/2}(\log m)^{(d-1)(r_1+r_2)+1/2} .
 $$
 \end{Theorem}
 \begin{proof} We remind some known results that we use in the proof. E. Belinsky (see  \cite{TrBe}) proved the following bounds
 \be\label{KolWb}
d_m(\bW^r_{2},L_\infty) \ll m^{-r}(\log m)^{(d-1)r+1/2}, \qquad r>1/2. 
\ee
The following bound was obtained in \cite{VT32} (see Theorem 3.1 there): For $g\in \bW^{a_1,a_2}_2$, $a_1>0$, $a_2>0 $ we have
\be\label{KolWg}
d_m(\bW^g_{2},L_2) = s_{m+1}(J_g) \ll m^{-a_1-a_2-1/2}(\log m)^{(d-1)(a_1+a_2)}  . 
\ee
We also need the following operators of fractional integration and differentiation. We begin with the univariate case.
In this subsection we discuss a slightly more general Bernoulli kernels and integral operators related to them (see \cite{VTbookMA}, Section 1.4). In the univariate case, for $a>0$, and $\al \in \bbR$ let
 $$
 F_{a,\al}(x):= 1+2\sum_{k=1}^\infty k^{-a}\cos (kx-\al\pi/2)
 $$
\be\label{Lb1}
  = 1+\sum_{k=1}^\infty k^{-a}(e^{i\al\pi/2}e^{-ikx}+e^{-i\al\pi/2}e^{ikx})
\ee
be the generalised Bernoulli kernel. Clearly, we have $F_{a}(x) = F_{a,a}(x)$, where $F_a(x)$ is defined in (\ref{Bi8}). Define the integral operator, acting on trigonometric polynomials $\phi(x)$, as
\be\label{Lb2}
(I^{(a,\al)}\phi)(x) := (I^{(a,\al)}_x\phi)(x) := ( F_{a,\al} \ast \phi)(x):= \frac{1}{2\pi}\int_{\bbT} F_{a,\al}(x-z) \phi(z)dz.
\ee
The operator $I^{(a,\al)}$ is the multiplier operator:
\be\label{Lb3}
(I^{(a,\al)}\phi)(x) = \hat \phi(0)+ \sum_{k<0}  |k|^{-a}e^{i\al\pi/2} \hat{\phi}(k)e^{ikx}+ \sum_{k>0}k^{-a}e^{-i\al\pi/2} \hat{\phi}(k)e^{ikx}.
\ee
Identity (\ref{Lb3}) implies that 
\be\label{Lb3'}
I^{(a,\al)}I^{(b,\bt)} = I^{(b,\bt)} I^{(a,\al)} = I^{(a+b,\al+\bt)} .
\ee

We now define the inverse operator to the operator $I^{(a,\al)}$, acting on the trigonometric polynomials from $\cT(2n)$ (we take $2n$ for convenience in the future use). Define
$$
\cD^{(a,\al)}_{2n}(x) := 1+2\sum_{k=1}^{2n} k^{a}\cos (kx+\al\pi/2)
$$
and the operator (for $h \in \cT(2n)$) 
\be\label{Lb4}
(D^{(a,\al)}h)(x) := (D^{(a,\al)}_xh)(x) := (\cD^{a,\al}_{2n} \ast h)(x):= \frac{1}{2\pi}\int_{\bbT} \cD^{a,\al}_{2n}(x-z) h(z)dz.
\ee
The operator $D^{(a,\al)}$ is the multiplier operator:
\be\label{Lb3a}
(D^{(a,\al)}h)(x) = \hat h(0)+ \sum_{k<0}  |k|^{a}e^{-i\al\pi/2} \hat{h}(k)e^{ikx}+ \sum_{k>0}k^{a}e^{i\al\pi/2} \hat{h}(k)e^{ikx}.
\ee
It is easy to see that for $h \in \cT(2n)$ we have
\be\label{Lb5}
I^{(a,\al)}D^{(a,\al)}h   = h. 
\ee
Clearly, the operator $D^{(a,\al)}$ can be defined for smooth enough functions $h$ instead of the trigonometric polynomials. 
Then relation (\ref{Lb5}) means that $I^{(a,\al)}D^{(a,\al)} =Id$, where $Id$ is the identity operator. 

For convenience, we write
$$
I^a_x := I^{(a,a)}_x,\qquad D^a_x := D^{(a,a)}_x.
$$
For  vectors $\br = (r_1,\dots,r_d)$ and $\bx = (x_1,\dots,x_d)$ define 
$$
I^\br_\bx := \prod_{j=1}^d I^{r_j}_{x_j},\qquad D^\br_\bx := \prod_{j=1}^d D^{r_j}_{x_j}. 
$$

Assume that $K\in \bW^\br_2 = \bW^{(\br^1,\br^2)}_2$ (see below) with $r_1>1/2$, $r_2>0$. Let $u$ be a number satisfying $1/2 <u< r_1$. 
This means that there exists $\phi \in L_{2}(\bbT^{2d})$, $\|\phi\|_2 \le 1$, such that
$$
K = I^{\br^1}_\bx I^{\br^2}_\by \phi,\qquad \br^1 =(r_1,\dots,r_1) \in \bbR^d,\quad  \br^2 =(r_2,\dots,r_2) \in \bbR^d.
$$
Represent 
$$
I^{\br^1}_\bx = I^{\bu}_\bx I^{\ba}_\bx,\quad \bu := (u,\dots,u) \in \bbR^d,\quad \ba := (r_1-u,\dots,r_1-u) \in \bbR^d. 
$$ 
Denote $g:= I^{\ba}_\bx I^{\br^2}_\by \phi$. Then $g\in \bW^{(\ba,\br^2)}_2$ and $J_K = J_{F_\bu}J_g$. Therefore,
\be\label{Lb6}
d_{2m}(\bW^K_2,L_\infty) \le d_m(\bW^\bu_2,L_\infty) d_{m}(\bW^g_2,L_2).
\ee
 We now use relations (\ref{KolWb}), (\ref{KolWg}) and complete the proof. 

 \end{proof}
 
 For the future use we formulate the inequality (\ref{Lb6}) proved above as a separate statement.
 
 \begin{Lemma}\label{KBL2} Let $d\in \bbN$ and $\bu \in \bbR^d_+$, $\bu := (u,\dots,u)$, $u>1/2$. Assume that $K$ is such that $g := D^\bu_\bx K \in L_2(\Omega^1 \times \Omega^2)$. Then 
 \be\label{Lb7}
d_{2m}(\bW^K_2,L_\infty) \le d_m(\bW^\bu_2,L_\infty) d_{m}(\bW^g_2,L_2).
\ee
 \end{Lemma}
 
 We now prove an analog of Theorem \ref{KBT1} for the $\bH$ classes.
 
  \begin{Theorem}\label{KBT2} Let $d\in\bbN$. 
Assume that  we have $r_1>1/2$, $r_2>0$. Then
 $$
\sup_{K\in \bH^{r_1,r_2}_2}  d_m(\bW^K_2)_\infty \ll m^{-r_1-r_2-1/2}(\log m)^{(d-1)(r_1+r_2)+d-1/2} .
 $$
 \end{Theorem} 
 \begin{proof} We use notations from the above proof of Theorem \ref{KBT1}. By Lemma \ref{KBL2} we have 
  \be\label{Lb8}
d_{2m}(\bW^K_2,L_\infty) \le d_m(\bW^\bu_2,L_\infty) d_{m}(\bW^g_2,L_2), \quad g:= D^\bu_\bx K.
\ee

We now need one more lemma. 

\begin{Lemma}\label{KBL3} Assume that  we have $r_1>1/2$, $r_2>0$. Then for $K \in \bH^{r_1,r_2}_2$ and 
 $\bu := (u,\dots,u)$, $1/2<u< r_1$ we have $g := D^\bu_\bx K \in \bH^{r_1-u,r_2}_2$.
 \end{Lemma}
 \begin{proof} By Theorem \ref{H} we get for all $\bs \in \bbN_0^{2d}$
 $$
 \|A_\bs(K)\|_2 \ll 2^{-r_1\|\bs^1\|_1-r_2\|\bs^2\|_1}, \quad \bs^1:= (s_1,\dots,s_d), \quad   \bs^2:= (s_{d+1},\dots,s_{2d}).
 $$
 From here we easily obtain  
 $$
 \|A_\bs(g)\|_2 = \|A_\bs(D^\bu_\bx K)\|_2 \ll 2^{u\|\bs^1\|_1} \|A_\bs(K)\|_2 \ll 2^{-(r_1-u)\|\bs^1\|_1-r_2\|\bs^2\|_1}.
 $$
 By Theorem \ref{H} we conclude $g := D^\bu_\bx K \in \bH^{r_1-u,r_2}_2$, which proves Lemma \ref{KBL3}.
 \end{proof} 
 
 We continue proof of Theorem \ref{KBT2}. By (\ref{KolWb}) we get
 \be\label{KB9a}
 d_m(\bW^\bu_2,L_\infty) \ll m^{-u} (\log m)^{(d-1)u +1/2}.
 \ee
 By Corollary \ref{KBC1} and Remark \ref{KBR1} we have for $ g \in  \bH^{r_1-u,r_2}_2$
 \be\label{KB10a}
 d_{m}(\bW^g_2,L_2) \ll m^{-(r_1-u)-r_2-1/2} (\log m)^{(d-1)(r_1-u+r_2) +(d-1)}.
 \ee
 Combining (\ref{Lb8}) -- (\ref{KB10a}), we complete the proof of Theorem \ref{KBT2}.
 \end{proof}
 
 \section{Some relations between different asymptotic characteristics}
 \label{RI}
 
  \subsection{Connections between linear recovery and nonlinear approximations}
\label{RN}

{\bf Sampling recovery.} We begin with a simple inequality for the linear recovery $\ro_m(\bW^K_q,L_p)$. 
 
 \begin{Proposition}\label{RNP1} Let $1\le q,p, \le \infty$. Assume that for every $\bz\in \Omega^1$ we have $K(\bz,\cdot) \in L_{q'}(\Omega^2)$, $q':=q/(q-1)$. Then we have
 \be\label{RN1}
 \ro_m(\bW^K_q,L_p) \le \sigma_m(K,\cL\cK)_{p,q'}.
 \ee
 \end{Proposition}
\begin{proof} Consider an operator $\Psi_m$ of linear recovery
$$
\Psi_m(f,\xi,\bx) := \sum_{j=1}^m f(\xi^j)\psi_j(\bx).
$$
Then we have for $f\in \bW^K_q$
$$
\|f-\Psi_m(f,\xi,\bx)\|_p \le \int_{\Omega^2} \left\|K(\cdot,\by)-\sum_{j=1}^m K(\xi^j,\by)\psi_j(\cdot)\right\|_p |\ff(\by)|d\mu_2.
$$
This implies that 
$$
\sup_{f\in\bW^K_q} \|f-\Psi_m(f,\xi,\bx)\|_p \le  \left\|K(\bx,\by)-\sum_{j=1}^m K(\xi^j,\by)\psi_j(\bx)\right\|_{p,q'}  .
$$
We now take infimum over sets of points $\{\xi^j\}_{j=1}^m$ and sets of functions $\{\psi_j\}_{j=1}^m$ and complete the proof. 

\end{proof}

For the next simple relation we need a new notation. Define for $p_1,p_2$
$$
\|f(\bx,\by)\|_{L^*_{p_1,p_2}}:=\|f(\bx,\by)\|^*_{p_1,p_2} := \|\|f(\bx,\cdot)\|_{p_2}\|_{p_1},
$$
which means that first we take the norm with respect to $\by$ and after that the norm with respect to $\bx$. 

 \begin{Proposition}\label{RNP2} Let $1\le q \le \infty$. Assume that for every $\bz\in \Omega^1$ we have $K(\bz,\cdot) \in L_{q'}(\Omega^2)$, $q':=q/(q-1)$. Then we have
 \be\label{RN2}
 \ro_m(\bW^K_q,L_\infty) = \sigma_m(K,\cL\cK)_{L^*_{\infty,q'}}.
 \ee
 \end{Proposition}
\begin{proof} In the same way as in the above proof of Proposition \ref{RNP1} we obtain for $f\in \bW^K_q$
$$
\|f-\Psi_m(f,\xi,\bx)\|_\infty 
$$
$$
=\sup_{\bx\in\Omega^1}\left| \int_{\Omega^2} \left(K(\bx,\by)-\sum_{j=1}^m K(\xi^j,\by)\psi_j(\bx)\right) \ff(\by)d\mu_2 \right| =: \sup_{\bx\in\Omega^1}E(\bx,\ff).
$$
Therefore,
$$
\sup_{f\in\bW^K_q}\|f-\Psi_m(f,\xi,\bx)\|_\infty= \sup_{f\in\bW^K_q}\sup_{\bx\in\Omega^1}E(\bx,\ff) = \sup_{\bx\in\Omega^1} \sup_{f\in\bW^K_q}E(\bx,\ff) 
$$
$$
=\left\|K(\bx,\by)-\sum_{j=1}^m K(\xi^j,\by)\psi_j(\bx)\right\|^*_{\infty,q'}.
$$
We now take infimum over sets of points $\{\xi^j\}_{j=1}^m$ and sets of functions $\{\psi_j\}_{j=1}^m$ and complete the proof. 
\end{proof}

 \subsection{Some inequalities}
 \label{2In}
 
 In the Section \ref{In} we defined the systems $\Pi=\cL\cL$, $\cL\cK$, $\cK\cL$, and $\cK\cK$. In this section we only discuss the best $m$-term approximations with respect to some of these systems and therefore normalization of elements of these systems does not play any role. Obviously, we have the following inclusions for any $\bp$
 $$
 \cK\cK(\bp) \subset \cL\cK(\bp) \subset \Pi(\bp).
 $$
 These inclusions immediately imply the following trivial inequalities
 $$
 \sigma_m(K,\Pi(\bp))_\bp \le   \sigma_m(K,\cL\cK(\bp))_\bp  \le \sigma_m(K,\cK\cK(\bp))_\bp. 
 $$
 
 In this section we discuss the following fundamental problems.
 
 {\bf Problem $\cL\cK-\Pi$.} Find an upper bound for $\sigma_m(K,\cL\cK(\bp))_\bp$ in terms of  $\sigma_n(K,\Pi(\bp))_\bp$
 with $n$ close to $m$.
 
  {\bf General Problem $\cL\cK-\Pi$.} Find an upper bound for $\sigma_m(K,\cL\cK(\bp))_\bp$ in terms of  $\sigma_n(K^*,\Pi(\bp))_\bp$ with $n$ close to $m$, where $K^*$ is a new function build from $K$. Certainly, we would like the operator 
  mapping $K$ to $K^*$ to be as simple as possible. For instance, it might be a differentiation operator, which is popular in approximation theory. 
  
  We begin with a result on the Problem $\cL\cK-\Pi$. 
  
  \begin{Theorem}\label{2InT1} For any $b\in(1,2]$ there exists a positive constant $B=B(b)$ such that for any continuous on $\Omega^1\times\Omega^2$ function $K(\bx,\by)$   we have
\be\label{2In1}
  \sigma_m(K,\cL\cK(\infty))_\infty \le Bm^{1/2}  \sigma_{\theta (m-1)}(K,\Pi(\infty))_\infty
  \ee
  with $\theta =1/b$ in the real case and $\theta =1/(2b)$ in the complex case. 
  \end{Theorem}
  \begin{proof} By Proposition \ref{RNP2} with $q=1$ we know that 
  \be\label{2In2}
   \sigma_m(K,\cL\cK(\infty))_\infty = \ro_m(\bW^K_1,L_\infty).
  \ee
  By Theorem \ref{krT1} with $\bF= \bW^K_1$  we obtain
  \be\label{2In3}
    \ro_m(\bW^K_1,L_\infty) \le Bm^{1/2} d_{\theta (m-1)}(\bW^K_1, L_\infty)
  \ee
  with $\theta =1/b$ in the real case and $\theta =1/(2b)$ in the complex case.
  Finally, by Lemma \ref{KBL1} with $p=\infty$ we get
   \be\label{2In4}
    d_{\theta (m-1)}(\bW^K_1, L_\infty) = \sigma_{\theta (m-1)}(K,\Pi(\infty))_\infty.
  \ee
  Combining relations (\ref{2In2}) -- (\ref{2In4}), we complete the proof of Theorem \ref{2InT1}. 
  \end{proof}
  
  \begin{Proposition}\label{2InP1} The extra factor $m^{1/2}$ in the inequality (\ref{2In1}) of Theorem \ref{2InT1} is sharp. 
  \end{Proposition}
  \begin{proof} The claim of Proposition \ref{2InP1} follows from known results.  The following result on bilinear approximations is known.  

 \begin{Theorem}[{\cite{VT32}}, Theorem 2]\label{InT5} Let $d=1$ and $\bF^\br_\bq$ denote one of the classes $\bW^\br_\bq$ or $\bH^\br_\bq$. Then for $\br > \mathbf{1}$ and $1\le q_1\le p_1 \le \infty$, $1\le q_2,p_2 \le \infty$ we have 
 $$
 \sup_{K\in \bF^\br_\bq} \sigma_m(K,\Pi(\bp))_\bp \asymp m^{-r_1-r_2 + \xi(q_1,p_1)}  
 $$
 where $\xi(q,p)$ is defined in (\ref{ksi}). 
 \end{Theorem}
 
 The following result of approximation with respect to adaptive dictionaries was obtained in the recent paper \cite{VT217} (see Theorem 1.8 there). 
 
  \begin{Theorem} \label{InT4b} Let $d=1$ and $\bF^\br_\bq$ denote one of the classes $\bW^\br_\bq$ or $\bH^\br_\bq$ (see the definition in Section \ref{fc} below) of functions of two variables.  Then for     $1\le q_1\le 2$, $1\le q_2 \le \infty$,  and $\br > \br(\bq)$   we have 
 $$
\sup_{K\in \bF^\br_\bq}  \sigma_m(K,\cL\cK(\infty))_\infty \asymp m^{-r_1-r_2 + 1/q_1} .
 $$
 \end{Theorem}

In the case of scalar $p  = \infty$ and scalar $q\in [1,2]$ we have 
$$
 \xi(q,p) = 1/q -1/2.
$$
It remains to compare Theorems \ref{InT5} and \ref{InT4b} in this case.
  
  \end{proof}
  
  We now proceed to the General Problem $\cL\cK-\Pi$ in the case of scalar $p=2$. 
  
  \begin{Theorem}\label{RIT2}  Let $\bu = (u,\dots,u)\in \bbR^d$, $u>1/2$. Assume that for every $\bz\in \Omega^1$ we have $K(\bz,\cdot) \in L_{2}(\Omega^2)$ and $K^{(u)} := D^{\bu}_\bx K \in L_2(\Omega^1\times\Omega^2)$. Then     we have
 \be\label{RN2a}
   \sigma_m(K,\cL\cK)_{L^*_{\infty,2}} \le C(u,d) m^{-u}(\log m)^{u(d-1)+1/2} \sigma_{m/8}(K^{(u)},\Pi)_2. 
 \ee
 \end{Theorem}
 
 Here is a direct corollary of Theorem \ref{RIT2}.
  \begin{Corollary}\label{RIC1} Under conditions of Theorem \ref{RIT2} we have 
   \be\label{RI2}
   \sigma_m(K,\cL\cK)_2 \le C(u,d) m^{-u}(\log m)^{u(d-1)+1/2} \sigma_{m/16}(K^{(u)},\Pi)_2.
 \ee
 \end{Corollary}
 {\bf Proof of Theorem \ref{RIT2}.} By Proposition \ref{RNP2} with $q=2$ we find that 
  \be\label{RI3}
   \sigma_m(K,\cL\cK)_{L^*_{\infty,2}} = \ro_m(\bW^K_2,L_\infty).
  \ee
  By Theorem \ref{KPUU2} with $\bF= \bW^K_2$ and $p=\infty$ we obtain
  \be\label{RI4}
    \ro_m(\bW^K_2,L_\infty) \le Cm^{1/2} d_{m/4}(\bW^K_2, L_\infty).
  \ee
By Lemma \ref{KBL2} we get
 \be\label{RI5}
d_{m/4}(\bW^K_2,L_\infty) \le d_{m/8}(\bW^\bu_2,L_\infty) d_{m/8}(\bW^g_2,L_2), \quad g:= K^{(u)}.
\ee
  By (\ref{KolWb}) we find
   \be\label{RI6}
d_{m/8}(\bW^u_{2},L_\infty) \ll m^{-u}(\log m)^{(d-1)u+1/2}, \qquad u>1/2. 
\ee
 Next by (\ref{KB7}) and  (\ref{KB9}) we conclude
 \be\label{RI7}
 d_{m/8}(\bW^g_2,L_2) \le (m/8)^{-1/2} \sigma_{m/16}(g,\Pi)_2.
 \ee
 Combining (\ref{RI3}) -- (\ref{RI7}), we complete the proof of Theorem \ref{RIT2}.
  
\subsection{Some upper bounds}
\label{UP}

We formulate corollaries of Proposition \ref{RNP2}, Theorems \ref{KBT1}, \ref{KBT2}, and Theorem \ref{krT1}. 

 \begin{Theorem}\label{UBT1} Let $d\in\bbN$. 
Assume that  we have $r_1>1/2$, $r_2>0$. Then
 $$
\sup_{K\in \bW^\br_2}  \sigma_m(K,\cL\cK)_{L^*_{\infty,2}} \ll m^{-r_1-r_2}(\log m)^{(d-1)(r_1+r_2)+1/2} .
 $$
 \end{Theorem}
 
 \begin{Corollary}\label{UBC1} Let $d\in\bbN$. 
Assume that  we have $r_1>1/2$, $r_2>0$. Then
 $$
\sup_{K\in \bW^\br_2}  \sigma_m(K,\cL\cK)_{2} \ll m^{-r_1-r_2}(\log m)^{(d-1)(r_1+r_2)+1/2} .
 $$
 \end{Corollary}
 
  \begin{Theorem}\label{UBT2} Let $d\in\bbN$. 
Assume that  we have $r_1>1/2$, $r_2>0$. Then
 $$
\sup_{K\in \bH^{r_1,r_2}_2}  \sigma_m(K,\cL\cK)_{L^*_{\infty,2}} \ll m^{-r_1-r_2}(\log m)^{(d-1)(r_1+r_2)+d-1/2} .
 $$
 \end{Theorem}

\begin{Corollary}\label{UBC2} Let $d\in\bbN$. 
Assume that  we have $r_1>1/2$, $r_2>0$. Then
 $$
\sup_{K\in \bH^{r_1,r_2}_2}  \sigma_m(K,\cL\cK)_{2} \ll m^{-r_1-r_2}(\log m)^{(d-1)(r_1+r_2)+d-1/2} .
 $$
 \end{Corollary}

 \Addresses
 
\end{document}